\documentclass[12pt]{amsart}

\usepackage{amsmath}
\usepackage{amssymb}
\usepackage{fancybox}
\usepackage{ascmac}
\usepackage{amsthm}
\usepackage{latexsym}
\usepackage[dvips]{graphicx}
\usepackage{graphics}

\newtheorem{theorem}{Theorem}[section]
\newtheorem{lemma}[theorem]{Lemma}

\newtheorem{proposition}[theorem]{Proposition}
\newtheorem{remark}[theorem]{Remark}

\newtheorem{conjecture}[theorem]{Conjecture}

\newcommand*{\seclab}[1]{\label{sec:#1}}
\newcommand*{\secref}[1]{Section~\ref{sec:#1}}

\newcommand*{\thmlab}[1]{\label{thm:#1}}
\newcommand*{\thmref}[1]{Theorem~\ref{thm:#1}}

\newcommand*{\proplab}[1]{\label{prop:#1}}
\newcommand*{\propref}[1]{Proposition~\ref{prop:#1}}

\newcommand*{\conjlab}[1]{\label{conj:#1}}
\newcommand*{\conjref}[1]{Conjecture~\ref{conj:#1}}

\newcommand*{\lemlab}[1]{\label{lem:#1}}
\newcommand*{\lemref}[1]{Lemma~\ref{lem:#1}}

\def\zz{{\mathbb Z}}

\def\cala{{\mathcal A}}
\def\calb{{\mathcal B}}

\def\cald{{\mathcal D}}

\def\calg{{\mathcal G}}

\def\cali{{\mathcal I}}
\def\calj{{\mathcal J}}
\def\calk{{\mathcal K}}

\def\calm{{\mathcal M}}

\def\calp{{\mathcal P}}
\def\calq{{\mathcal Q}}

\def\cals{{\mathcal S}}

\def\calw{{\mathcal W}}

\def\bfa{{\bf a}}
\def\bfb{{\bf b}}
\def\bfc{{\bf c}}
\def\bfd{{\bf d}}
\def\bfe{{\bf e}}

\def\bfn{{\bf n}}

\def\widedm{\widetilde{\cald}_M}
\def\wideDM#1{\widetilde{\cald}_{M_{#1}}}
\def\Bff{{\bf F}}

\def\BFF#1{{\bf F}_{#1}}
\def\bfn{{\rm N}}

\def\mqg{\calm_{\calq \calg}}
\def\mq{\calm_{\calq}}

\begin{document}

\title{Toric ideals of series and parallel connections of matroids}
\author{Kazuki Shibata}


\address{Kazuki Shibata,
Department of Mathematics,
College of Science,
Rikkyo University,
Toshima-ku, Tokyo 171-8501, Japan.}
\email{k-shibata@rikkyo.ac.jp}

\date{}

\begin{abstract}
In 1980, White conjectured that the toric ideal associated to a matroid is generated by binomials corresponding to a symmetric exchange.
In this paper, we prove that classes of matroids for which the toric ideal is generated by quadrics and that has quadratic Gr\"obner bases, are closed under series and parallel extensions, series and parallel connections, and 2-sums.
\end{abstract}

\maketitle

\section{Introduction}
A matroid has multiple equivalent definitions.
We define a matroid as a collection of subsets that satisfies the {\it exchange axiom}:
A {\it matroid} $M$ is a pair $(E , \calb)$, where $E=[d]=\{ 1,\ldots , d \}$ and $\calb$ is a non-empty collection of subsets of $E$, which satisfies
\begin{itemize}
\item for every $B$ and $B^{'}$ in $\calb$, for any $x \in B \setminus B^{'}$, there exists $y \in B^{'} \setminus B$ such that $(B \cup \{ y \}) \setminus \{ x \}$ is a member of $\calb$. 
\end{itemize}
The exchange axiom is equivalent to the following stronger axiom known as the {\it symmetric exchange property}.
\begin{itemize}
\item For every $B$ and $B^{'}$ in $\calb$, for any $x \in B \setminus B^{'}$, there exists $y \in B^{'} \setminus B$ such that $(B \cup \{ y \}) \setminus \{ x \}$ and $(B^{'} \cup \{ x \}) \setminus \{ y \}$ are members of $\calb$.
\end{itemize}

For a detailed introduction to matroid theory, see \cite{Oxley}.
We call a member of $\calb$ a {\it basis} of $M$.
All the members of $\calb$ have the same cardinality.
This cardinality is said to be the {\it rank} of $M$ and is denoted by ${\rm rk}(M)$.
Let $\calb (M) = \{ B_1 ,\ldots, B_n \}$ be the collection of bases of a matroid $M$ on $E$ and $\cald_M = (\bfb_1 ,\ldots , \bfb_n)$ be an integer matrix satisfying $\bfb_j = \sum_{l \in B_j} \bfe_l$, where $\bfe_l$ is the $l$-th standard vector.
Let $K$ be a field, and let $K[X] = K[x_{1},\ldots , x_{n}]$ be the polynomial ring over $K$.
We consider the ring homomorphism
$$
\pi_M : K[X] \rightarrow K[S] = K[s_1 ,\ldots , s_d], \quad
x_{j} \mapsto S^{\bfb_j},
$$
where $S^{\bfb_j} = s^{b_1}_{1} \cdots s^{b_d}_{d}$ for $\bfb_j = (b_1,\ldots,b_d)$.
The toric ideal $J_{\cald_M}$, or briefly, $J_M$ is the kernel of $\pi_M$ (see \cite{Sturmfels2}).
The semigroup ring $R_M = K[X]/J_M$ is called the {\it bases monomial ring} of $M$, and it was introduced by N. White.
White proved that the bases monomial ring $R_M$ is normal and, in particular, it is a Cohen-Macaulay ring for any matroid $M$ (see \cite{White1}).
White presented the following conjecture.

\begin{conjecture}
[{\rm \cite{White2}}]
\conjlab{White}
{\rm For any matroid $M$, $J_M$ is generated by the binomials $x_i x_j - x_k x_l$ such that the pair of bases $B_k$, $B_l$ can be obtained from the pair of $B_i$, $B_j$ by a symmetric exchange.}
\end{conjecture}
It is natural to ask for the following variant of \conjref{White}.

\begin{conjecture}
\conjlab{WhiteGB}
{\rm For any matroid $M$, $J_M$ has a Gr\"obner basis consisting of the binomials $x_i x_j - x_k x_l$ such that the pair of bases $B_k$, $B_l$ can be obtained from the pair of $B_i$, $B_j$ by a symmetric exchange.}
\end{conjecture}

\conjref{White} is true for graphic matroids \cite{Blasiak}, matroids with ${\rm rank} \le 3$ \cite{Kashiwabara}, sparse paving matroids \cite{Bonin} and strongly base orderable matroids \cite{LaMich}.
\conjref{WhiteGB} is true for uniform matroids \cite{Sturmfels}, matroids with ${\rm rank} \le 2$ \cite{Blum,OhsugiHibi}, graphic matroids without $M(K_4)$-minor \cite{Blum} and lattice path matroids \cite{Sch}.

Let $\mqg$ be the class of matroids for which the toric ideal $J_M$ has a quadratic Gr\"obner basis, let $\mq$ be the class of matroids for which $J_M$ is generated by quadrics, and let $\calm$ be the class of all matroids.
The following is weaker than \conjref{White} and \conjref{WhiteGB}.
However it is still open.

\begin{conjecture}
\conjlab{Whiteweak}
{\rm The equalities $\mqg = \mq = \calm$ hold.}
\end{conjecture}
Conca proved that the class of transversal polymatroids is contained in $\mq$ \cite{Conca}.

Let $M$ be a matroid on $E$, and let $\calb(M)$ be the collection of bases of $M$.
An element $i \in E$ is called a {\it loop} of $M$ if it does not belong to any basis of $M$.
Dually, an element $i \in E$ is said to be a {\it coloop} of $M$ if it is contained in all the bases of $M$.
Let
$
\calb^{\ast} (M) = \{ E \setminus B ~|~ B \in \calb (M) \}.
$
Then a pair $(E,\calb^{\ast} (M))$ is a matroid, and it is called the {\it dual} of $M$ and is denoted as $M^{\ast}$.

Let $M$ and $\calb(M)$ be as above, and let $c \in E$.
We consider the following collection of subsets of $E\setminus \{ c \}$:
$$
\calb(M) \setminus c =
\begin{cases}
\{ B \setminus \{ c \} ~|~ B \in \calb(M) \} & \text{if~$c$~is~a~coloop~of~$M$,} \\
\{ B ~|~ c \notin B \in \calb(M) \} & \text{otherwise.}
\end{cases}
$$
A pair $(E \setminus \{ c \} , \calb(M) \setminus c)$ is a matroid, and it is called the {\it deletion} of $c$ from $M$ and is denoted as $M \setminus c$.
Dually, let $M/c$, the {\it contraction} of $c$ from $M$, be given by $M/c=(M^{\ast} \setminus c)^{\ast}$.
We call a matroid $M^{'}$ a {\it minor} of a matroid $M$ if $M^{'}$ can be obtained from $M$ by a finite sequence of contractions and deletions.

Let $M_1$ and $M_2$ be matroids with $E_1 \cap E_2 = \emptyset$.
Let $\calb(M_1)$ and $\calb(M_2)$ be collections of bases of $M_1$ and $M_2$, and let
$$
\calb(M_1) \oplus \calb(M_2) = \{ B \cup D ~|~ B \in \calb(M_1) , D \in \calb(M_2) \}.
$$
Then a pair $(E, \calb(M_1) \oplus \calb(M_2))$, where $E=E_1 \cup E_2$, is a matroid.
This matroid is called the {\it direct sum} of $M_1$ and $M_2$, and it is denoted as
$M_1 \oplus M_2$.

\begin{proposition}
[{\rm \cite{Blum,White2}}]
\proplab{Blu}
Classes $\mqg$ and $\mq$ are closed under the duality, taking minors and direct sums.
\end{proposition}

The outline of this paper is as follows.
In Section 2, we describe how to compute generating sets and Gr\"obner bases for the toric ideal of series and parallel extensions.
In Section 3, we use the results from Section 2 to form generating sets and Gr\"obner bases for the toric ideal of series and parallel connections and 2-sums.

\section{A series and parallel extension of a matroid}
\seclab{0}

Let $M$ be a matroid on $E=[d]$, and let $\calb(M)$ be the collection of bases of $M$.
Then a {\it series extension} of $M$ at $c \in E$ by $d+1$ is a matroid on $E \cup \{ d+1 \}$ which has
$$
\{
B \cup \{ d+1 \} ~|~ B \in \calb(M)
\} \cup \{ B   \cup \{ c \} ~|~ c \notin B \in \calb (M)  \}
$$
as the collection of bases and is denoted as $M +_{c} (d+1)$.
Dually, we call a matroid $(M^{\ast} +_{c} (d+1))^{\ast}$ a {\it parallel extension} of $M$ at $c$ by $d+1$.
A {\it series-parallel extension} of $M$ is any matroid derived from $M$ by a finite sequence of series and parallel extensions.
We suppose that $M$ does not have $c \in E$ as a coloop.
Let $\calb(M) = \{ B_1 , \ldots , B_{\gamma} , \ldots , B_n \}$ be the collection of bases of $M$, where $c \notin B_j$ for $j \in [\gamma]$ and $c \in B_j$ for $j \in [n] \setminus [\gamma]$.
We renumber the bases of $M$, if necessary.
Let $\cald_M = ( \bfb_{1} ,\ldots , \bfb_{n} )$ denote an integer matrix arising from $M$.
Now we consider a new integer matrix
$$
\widedm = 
\left(
\begin{array}{ccc|ccc}
\bfb_1 & \cdots & \bfb_n & \bfb_1 & \cdots & \bfb_{\gamma} \\
\bfe_1 & \cdots & \bfe_1 & \bfe_2 & \cdots & \bfe_2
\end{array}
\right).
$$
We define the ring homomorphism $\widetilde{\pi}_M$ as follows:
\begin{eqnarray*}
\widetilde{\pi}_M : K[X]=K[x^{1}_{1} ,\ldots , x^{1}_{n} , x^{2}_{1} ,\ldots , x^{2}_{\gamma}] & \rightarrow & K[S,w_1,w_2], \\
x^{i}_{j} & \mapsto & S^{\bfb_{j}} w_i.
\end{eqnarray*}

Let $\omega \in \zz^{n}_{\ge 0}$, and let $\prec$ be an arbitrary monomial order.
We define a new monomial order ${\prec}_{\omega}$ as follows:
$$
X^{\bfa} {\prec}_{\omega} X^{\bfb} \Leftrightarrow
\begin{cases}
 \omega \cdot \bfa  <  \omega \cdot \bfb  ~; ~\text{{\rm or}} \\
 \omega \cdot \bfa  =  \omega \cdot \bfb  ~\text{{\rm and}}
~ X^{\bfa} \prec X^{\bfb},
\end{cases}
$$
for $\bfa , \bfb \in \zz^{n}_{\ge 0}$.
We call a monomial order ${\prec}_{\omega}$ a {\it weight order} on $K[x_1,\ldots , x_n]$.
We use the following useful result:
\begin{proposition}
[{\rm \cite[Proposition 1.11]{Sturmfels}}]
For any monomial order $\prec$ and any ideal $I \subset K[x_1,\ldots,x_n]$, there exists a vector $w \in \zz^{n}_{\ge 0}$ such that ${\rm in}_{\omega} (I) = {\rm in}_{\prec} (I)$.
\end{proposition}

Let $\Bff$ be a homogeneous generating set for $J_{\cald_M}$, and let
$$
f = 
\underline{\prod_{l=1}^{u_f} x^{1}_{j_l} 
\prod_{l=1}^{v_f} x^{1}_{k_l}
}-
\prod_{l=1}^{u^{'}_f} x^{1}_{j^{'}_l}
\prod_{l=1}^{v^{'}_f} x^{1}_{k^{'}_l}
\in \Bff ,
$$
where $j_l , j^{'}_{l} \in [\gamma], k_l , k^{'}_l \in [n] \setminus [\gamma]$.
However, if $v_f \neq v^{'}_f$, then $\pi_M (f) \ne 0$ since the $c$-th entry of $\sum_{l=1}^{v_f} \bfb_{k_l}$ does not coincide with the $c$-th entry of $\sum_{l=1}^{v^{'}_f} \bfb_{k^{'}_l}$, and the c-th entries of $\sum_{l=1}^{u_f} \bfb_{j_l}$ and $\sum_{l=1}^{u^{'}_f} \bfb_{j^{'}_l}$ are zero.
Therefore $u_f = u^{'}_{f}$ and $v_f = v^{'}_{f}$.
Now let $I = (i_1 ,\ldots , i_{u_f}) \in \{ 1,2 \}^{u_f}$ and consider the binomial $f^{I} \in K[X]$ defined by
$$
f^{I} =
\underline{\prod_{l=1}^{u_f} x^{i_l}_{j_l} 
\prod_{l=1}^{v_f} x^{1}_{k_l}
}-
\prod_{l=1}^{u_f} x^{i_l}_{j^{'}_l}
\prod_{l=1}^{v_f} x^{1}_{k^{'}_l}
.
$$
Since $f \in J_{\cald_M}$, the new homogeneous binomial $f^{I} \in J_{\widedm}$.
We set
$$
\widetilde{\Bff} = \{ f^{I} ~|~ f \in \Bff , I \in \{ 1,2 \}^{u_f} \} \cup \{ \underline{x^{1}_{j_2} x^{2}_{j_1}} - x^{1}_{j_1} x^{2}_{j_2} ~|~ 1 \le j_1 < j_2 \le \gamma \}.
$$

\begin{lemma}
\lemlab{main1}
Let $M$ be a matroid on $E$, and let
$\Bff$ be a Gr\"obner basis for $J_{\cald_M}$.
Then $\widetilde{\Bff}$ is a Gr\"obner basis for $J_{\widedm}$.
\end{lemma}

\begin{proof}
First, it is easy to see that $\widetilde{\Bff} \subset J_{\widedm}$.
Let $\omega = (\omega^{1}_{1},\ldots,\omega^{1}_{n})$ be a weight vector.
We denote the underlined monomial of $f$ as the initial monomial of $f$ with respect to a weight order induced by $\omega$.
Let $\widetilde{\omega}=(\omega^{1}_{1},\ldots,\omega^{1}_{n} ,\omega^{2}_{1} , \ldots , \omega^{2}_{\gamma})$
denote a weight vector satisfying $\omega^{1}_{j} = \omega^{2}_{j}$ for $j \in [\gamma]$.
Then the underlined monomial of $f^{I}$ is the initial monomial of $f^{I}$ with respect to a weight order $\prec_{\widetilde{\omega}}$.
We choose a tie-breaking monomial order on $K[X]$ that makes the monomial $x^{1}_{j_2} x^{2}_{j_1}$ for $1 \le j_1 < j_2 \le \gamma$ the initial monomial.
Let ${\rm in}(\Bff) = \langle {\rm in}_{\omega}(f) ~|~ f \in \Bff \rangle$ and ${\rm in}(\widetilde{\Bff}) = \langle {\rm in}_{\prec_{\widetilde{\omega}}}(f) ~|~ f \in \widetilde{\Bff} \rangle$.
Let $u$ and $v$ be monomials which are not in ${\rm in}(\widetilde{\Bff})$:
\begin{eqnarray*}
u &=& \prod_{l=1}^{m_1} (x^{1}_{i_{l}})^{p_l}
\prod_{l=1}^{m_2} (x^{2}_{j_{l}})^{q_l}
\prod_{l=1}^{m_3} (x^{1}_{k_{l}})^{r_l} \\
v &=& \prod_{l=1}^{m^{'}_1} (x^{1}_{i^{'}_{l}})^{p^{'}_l}
\prod_{l=1}^{m^{'}_2} (x^{2}_{j^{'}_{l}})^{q^{'}_l}
\prod_{l=1}^{m^{'}_{3}} (x^{1}_{k^{'}_{l}})^{r^{'}_l},
\end{eqnarray*}
where $p_l$, $q_l$, $r_l$, $p^{'}_l$, $q^{'}_l$, $r^{'}_l \in \zz_{>0}$ for any $l$, and $\cali = \{ i_1 , \ldots , i_{m_1} \}$, $\cali^{'} = \{ i^{'}_1 , \ldots , i^{'}_{m^{'}_1} \}$, $\calj = \{ j_1 , \ldots , j_{m_2} \}$, and $\calj^{'} = \{ j^{'}_1 , \ldots , j^{'}_{m^{'}_2} \}$ are subsets of $[\gamma]$ with cardinalities $m_1$, $m^{'}_1$, $m_2$, and $m^{'}_2$, respectively; and $\calk = \{ k_1 , \ldots , k_{m_3} \}$ and $\calk^{'} = \{ k^{'}_1 , \ldots , k^{'}_{m^{'}_3} \}$ are subsets of $[n] \setminus [\gamma ]$ with cardinalities $m_3$ and $m^{'}_3$, respectively.
Since neither $u$ nor $v$ is divided by $x^{1}_{j_2} x^{2}_{j_1}$ for $1 \le j_1 < j_2 \le \gamma$, it follows that $i_l \le j_{l^{'}}$ for $l \in [m_1]$ and $l^{'} \in [m_2]$,  and $i^{'}_{l} \le j^{'}_{l^{'}}$ for $l \in [m^{'}_1]$ and $l^{'} \in [m^{'}_2]$.
We suppose that $\widetilde{\pi}_M (u) = \widetilde{\pi}_M (v)$, i.e., 
$p+q=p^{'}+q^{'}$, $r=r^{'}$ and
$$
\sum_{l=1}^{m_1} p_l \bfb_{i_l} +
\sum_{l=1}^{m_2} q_l \bfb_{j_l} +
\sum_{l=1}^{m_3} r_l \bfb_{k_l} =
\sum_{l=1}^{m^{'}_1} p^{'}_l \bfb_{i^{'}_l} +
\sum_{l=1}^{m^{'}_2} q^{'}_l \bfb_{j^{'}_l} +
\sum_{l=1}^{m^{'}_3} r^{'}_l \bfb_{k^{'}_l}.
$$
Here we set $p=\sum_{l=1}^{m_1} p_l$, $q=\sum_{l=1}^{m_2} q_l$, $r=\sum_{l=1}^{m_3} r_l$, $p^{'}=\sum_{l=1}^{m^{'}_1} p^{'}_l$, $q^{'}=\sum_{l=1}^{m^{'}_2} q^{'}_l$, and $r^{'}=\sum_{l=1}^{m^{'}_3} r^{'}_l$.
Therefore it follows that $\pi_M (u^{'}) = \pi_M (v^{'})$, where
\begin{eqnarray*}
u^{'} &=& \prod_{l=1}^{m_1} (x^{1}_{i_{l}})^{p_l}
\prod_{l=1}^{m_2} (x^{1}_{j_{l}})^{q_l}
\prod_{l=1}^{m_3} (x^{1}_{k_{l}})^{r_l}, \\
v^{'} &=& \prod_{l=1}^{m^{'}_1} (x^{1}_{i^{'}_{l}})^{p^{'}_l}
\prod_{l=1}^{m^{'}_2} (x^{1}_{j^{'}_{l}})^{q^{'}_l}
\prod_{l=1}^{m^{'}_3} (x^{1}_{k^{'}_{l}})^{r^{'}_l}.
\end{eqnarray*}
Hence $u^{'} - v^{'}$ belongs to $J_{\cald_M}$.
If $u^{'}$ or $v^{'}$ belongs to ${\rm in} (\Bff)$, then $u^{'}$ or $v^{'}$ is in ${\rm in} (\widetilde{\Bff})$.
In particular, $u$ or $v$ belongs to ${\rm in}(\widetilde{\Bff})$.
This is a contradiction.
Therefore neither $u^{'}$ nor $v^{'}$ belongs to ${\rm in} (\Bff)$.
Since $\Bff$ is a Gr\"obner basis for $J_{\cald_M}$, it follows that $u^{'} = v^{'}$.
In particular, $\cali = \cali^{'}$, $\calj = \calj^{'}$, $\calk = \calk^{'}$, $p_l = p^{'}_l$, $q_l = q^{'}_l$, and $r_l = r^{'}_l$ for any $l$.
Thus $u=v$.
Therefore $\widetilde{\Bff}$ is a Gr\"obner basis for $J_{\widedm}$.
\end{proof}

\begin{lemma}
\lemlab{gen}
Let $M$ be a matroid on $E$.
If $\Bff$ is a homogeneous generating set for $J_{\cald_M}$, then $\widetilde{\Bff}$ is a generating set for $J_{\widedm}$.
\end{lemma}

\begin{proof}
We assume that $\Bff$ and $\Bff^{'}$ are generating sets for $J_{\cald_M}$.
Then $\widetilde{\Bff}$ and $\widetilde{\Bff^{'}}$ generate the same ideal.
In particular, this holds if $\Bff^{'}$ is a Gr\"obner basis for $J_{\cald_M}$.
Thus $\langle \widetilde{\Bff} \rangle = \langle \widetilde{\Bff^{'}} \rangle$.
By \lemref{main1}, if $\Bff^{'}$ is a Gr\"obner basis for $J_{\cald_M}$, then $\widetilde{\Bff^{'}}$ is a generating set for $J_{\widedm}$, since $\widetilde{\Bff^{'}}$ is a Gr\"obner basis for $J_{\widedm}$.
\end{proof}

\begin{theorem}
\thmlab{sepa}
Let $M$ be a matroid on $E$, and let $M +_{c} (d + 1)$ denote a series extension of $M$ at $c$ by $d+1$.
Then $\widetilde{\Bff}$ becomes a generating set (resp. a Gr\"obner basis) for $J_{M +_{c} (d+1)}$. 
\end{theorem}

\begin{proof}
By elementary row operations on $\widetilde{\cald}_M$, we obtain the integer matrix arising from $M +_{c} (d+1)$.
\end{proof}

\begin{remark}
{\rm If $c$ is a coloop of $M$, then $J_{M +_{c} (d+1)} = J_M$.}
\end{remark}

\begin{theorem}
Classes $\mqg$ and $\mq$ are closed under series and parallel extensions.
\end{theorem}

\section{A series and parallel connection of matroids}

Let $M_1$ and $M_2$ be matroids with $E_1 \cap E_2 = \{ c \}$ and $E = E_1 \cup E_2$.
Suppose that for both $M_1$ and $M_2$, $c$ is neither a loop nor a coloop.
Let
\begin{eqnarray*}
\calb_{\cals} &=& \{ B \cup D ~|~ B \in \calb(M_1), D \in \calb(M_2) , B \cap D = \emptyset \}, \\
\calb_{\calp} &=& \{ B \cup D ~|~ B \in \calb(M_1), D \in \calb(M_2) , c \in B \cap D \} \\
& & \cup \{ (B \cup D) \setminus \{ c \} ~|~ B \in \calb(M_1) , D \in \calb(M_2) , \text{$c$~is~in~exactly~one~of~$B$~and~$D$} \}.
\end{eqnarray*}
Then pairs $(E,\calb_{\cals})$ and $(E,\calb_{\calp})$ are matroids.
These matroids are said to be the {\it series} and {\it parallel} connections of $M_1$ and $M_2$ with respect to the basepoint $c$.
We denote them as $S((M_1;c) , (M_2;c))$ and $P((M_1;c) , (M_2;c))$, or briefly, $S(M_1 , M_2)$ and $P(M_1 , M_2)$ \cite[Proposition 7.1.13]{Oxley}.

On the other hand, when $c$ is a loop of $M_1$, then we define
$$
P(M_1 , M_2) = M_1 \oplus (M_2 / c) \quad \text{{\rm and}} \quad
S(M_1 , M_2) = (M_1 / c) \oplus M_2.
$$
When $c$ is a coloop of $M_1$, then we define
$$
P(M_1 , M_2) = (M_1 \setminus c) \oplus M_2 \quad \text{{\rm and}} \quad
S(M_1 , M_2) = M_1 \oplus (M_2 \setminus c)
$$
(see \cite[7.1.5 - 7.1.8]{Oxley}).
Moreover, the 2-{\it sum} $M_1 \oplus_{2} M_2$ of $M_1$ and $M_2$ is $S(M_1 , M_2) / c$, or equivalently, $P(M_1 , M_2) \setminus c$, where $c$ is neither a loop nor a coloop of either $M_1$ or $M_2$.

Let $M_1$ and $M_2$ be matroids on $E_1 = [d_1]$ and $E_2= [d_2]$.
We identify the set $[d_2]$ with the set $\{ d_1 + 1,\ldots , d_1 +d_2\}$.
Assume that $c_i \in E_i$ is not a coloop of $M_i$ for $i \in [2]$.
Let
$$
\calb(M_1) = \{ B_1 , \ldots , B_{\gamma_1} , \ldots, B_{n_1} \} 
\quad \text{{\rm and}} \quad
\calb(M_2) = \{ D_1 , \ldots , D_{\gamma_2} , \ldots, D_{n_2} \}
$$
be collections of bases of $M_1$ and $M_2$, where $c_1 \notin B_j$ for $j \in [\gamma_1]$ and $c_2 \notin D_k$ for $k \in [\gamma_2]$.
Let $\cald_{M_1} = ( \bfb_{1} ,\ldots, \bfb_{n_1} )$ and $\cald_{M_2} = ( \bfd_{1} ,\ldots, \bfd_{n_2} )$ be two integer matrices arising from $M_1$ and $M_2$.
We define ring homomorphisms $\pi_{M_1}$ and $\pi_{M_2}$ by setting
\begin{eqnarray*}
\pi_{M_1} : K[x^{1}_{1},\ldots, x^{1}_{n_1}] & \rightarrow & K[S] \qquad x^{1}_{j} \mapsto S^{\bfb_{j}} , \\
\pi_{M_2} : K[y^{2}_{1} ,\ldots , y^{2}_{n_2}] & \rightarrow & K[T] \qquad y^{2}_{k} \mapsto T^{\bfd_{k}}.
\end{eqnarray*}
Similar to what we did in \secref{0}, we consider two integer matrices
\begin{eqnarray*}
\wideDM{1} &=& \left(
\begin{array}{ccc|ccc}
\bfb_1 & \cdots & \bfb_{n_1} & \bfb_1 & \cdots & \bfb_{\gamma_1} \\
\bfe_1 & \cdots & \bfe_1 & \bfe_2 & \cdots & \bfe_2
\end{array} \right) ,
\\
\wideDM{2} &=& \left(
\begin{array}{ccc|ccc}
\bfd_1 & \cdots & \bfd_{{\gamma}_2} & \bfd_1 & \cdots & \bfd_{n_2} \\
\bfe_1 & \cdots & \bfe_1 & \bfe_2 & \cdots & \bfe_2
\end{array} \right).
\end{eqnarray*}
We define ring homomorphisms $\widetilde{\pi}_{M_1}$ and $\widetilde{\pi}_{M_2}$ as follows:
\begin{eqnarray*}
\widetilde{\pi}_{M_1} : K[X]=K[x^{1}_{1},\ldots, x^{1}_{n_1} , x^{2}_{1} ,\ldots , x^{2}_{\gamma_1}] & \rightarrow & K[S,w_1,w_2], \qquad x^{i}_{j} \mapsto S^{\bfb_{j}} w_{i}, \\
\widetilde{\pi}_{M_2} : K[Y]=K[y^{1}_{1} ,\ldots , y^{1}_{\gamma_2} , y^{2}_{1} , \ldots , y^{2}_{n_2}] & \rightarrow & K[T,w_1,w_2], \qquad y^{i}_{k} \mapsto T^{\bfd_{k}} w_{i}.
\end{eqnarray*}
Moreover, consider the integer matrix
$$
\widetilde{\cald} =
\left(
\begin{array}{ccc|c|ccc|ccc|c|ccc}
\bfb_1 & & \bfb_1 && \bfb_{n_1} && \bfb_{n_1}&\bfb_1&&\bfb_1&&\bfb_{\gamma_1}&&\bfb_{\gamma_1} \\
\bfd_1 &\cdots &\bfd_{\gamma_2} & \cdots & \bfd_1& \cdots & \bfd_{\gamma_2}& \bfd_1& \cdots& \bfd_{n_2}& \cdots & \bfd_1 & \cdots & \bfd_{n_2} \\
\bfe_1 & & \bfe_1 && \bfe_1 & & \bfe_1 & \bfe_2&& \bfe_2&& \bfe_2 & & \bfe_2 
\end{array}
\right).
$$
Let $K[Z]= K[z^{i}_{jk} ~|~ i=1,2, j \in [\alpha_i] , k \in [\beta_i]]$ be the polynomial ring over $K$, where $(\alpha_1 , \alpha_2) = (n_1 , \gamma_1)$ and $(\beta_1 , \beta_2) = (\gamma_2 , n_2)$.
The ring homomorphism $\widetilde{\pi}$ is defined by
$$
\widetilde{\pi} : K[Z]  \rightarrow  K[S,T,w_1,w_2], \qquad
z^{i}_{jk}  \mapsto  S^{\bfb_{j}} T^{\bfd_{k}} w_{i}.
$$

Let $\BFF{1}$ and $\BFF{2}$ be homogeneous generating sets for $J_{\cald_{M_1}}$ and $J_{\cald_{M_2}}$, respectively.
Then we define $\widetilde{{\bf F}}_1$ and $\widetilde{{\bf F}}_2$ in a way analogous to what we did in \secref{0}.
Let
$$
f = \prod_{l=1}^{u_f} x^{i_l}_{j^{1}_l} -
\prod_{l=1}^{u_f} x^{i_l}_{j^{2}_l} \in \widetilde{{\bf F}}_1,
$$
and let $k=(k_1 , \ldots , k_{u_f})$, with $k_l \in [\beta_{i_l}]$ for $1 \le l \le u_f$.
We consider the binomial $f_k \in K[Z]$ defined by
$$
f_k = \prod_{l=1}^{u_f} z^{i_l}_{j^{1}_l k_l} -
\prod_{l=1}^{u_f} z^{i_l}_{j^{2}_l k_l}.
$$
Since $f \in J_{\widetilde{\cald}_{M_1}}$, the new homogeneous binomial $f_k \in J_{\widetilde{\cald}}$.
If $\widetilde{{\bf F}}_1$ is any set of binomials in $J_{\widetilde{\cald}_{M_1}}$, then 
$$
{\rm Lift}(\widetilde{{\bf F}}_1) =
\left\{
f_k ~\left|~ f \in \widetilde{{\bf F}}_1, k \in \prod_{l=1}^{u_f} [\beta_{i_l}]
\right.
\right\}
.
$$
We define ${\rm Lift}(\widetilde{{\bf F}}_2)$ in an analogous way.
Furthermore, the quadratic binomial set ${\rm Quad}(\widetilde{\cald}_{M_1},\widetilde{\cald}_{M_2})$ is defined by
$$
{\rm Quad}(\widetilde{\cald}_{M_1},\widetilde{\cald}_{M_2}) = 
\left\{
z^{i}_{j_1 k_2} z^{i}_{j_2 k_1} -
z^{i}_{j_1 k_1} z^{i}_{j_2 k_2}
~\left|~
i=1,2 , 
\begin{array}{l}
1 \le j_1 < j_2 \le \alpha_i , \\
1 \le k_1 < k_2 \le \beta_i
\end{array}
\right.
\right\}.
$$
We set $\widetilde{\bfn}={\rm Lift}(\widetilde{{\bf F}}_1) \cup {\rm Lift}(\widetilde{{\bf F}}_2) \cup {\rm Quad}(\widetilde{\cald}_{M_1},\widetilde{\cald}_{M_2})$.

\begin{lemma}
\lemlab{main2}
Let $M_1$ and $M_2$ be matroids on $E_1 = [d_1]$ and $E_2=[d_2]$, respectively; and assume that $c_i \in E_i$ is not a coloop of $M_i$ for $i=1,2$.
Let $S(M_1,M_2)$ be a series connection of $M_1$ and $M_2$ with respect to the basepoint $c=c_1=c_2$.
Then
$$
N =\widetilde{\bfn} \cap
K[\widehat{Z}]
$$
is a generating set for $J_{S(M_1,M_2)}$.
Here we set $K[\widehat{Z}]=K[z^{i}_{jk} ~|~ i =1,2 , j \in [\alpha_i] , k \in V_i]$, where $V_1 = [\gamma_2]$ and $V_2 = [n_2] \setminus [\gamma_2]$.
Moreover, if $\BFF{1}$ and $\BFF{2}$ are Gr\"obner bases for $J_{\cald_{M_1}}$ and $J_{\cald_{M_2}}$, then there exists a monomial order such that $\bfn$ is a Gr\"obner basis for $J_{S(M_1,M_2)}$.
\end{lemma}

For the proof of \lemref{main2}, we use results in \cite{EKS,Sull}.

\noindent
Let $r>0$ be a positive integer, and let $\alpha, \beta \in \zz^{r}_{>0}$ be two vectors of positive integers.
Let
$$
K[X]=K[x^{i}_{j}~|~ i \in [r] , j \in [\alpha_i]] \quad \text{{\rm and}} \quad
K[Y]=K[y^{i}_{k}~|~ i \in [r] , k \in [\beta_i]]
$$
be two multigraded polynomial rings with the multigrading $\deg (x^{i}_{j}) = \deg (y^{i}_{k}) = \bfa^{i} \in \zz^{d}$, where $\alpha_i$ (resp. $\beta_i$) is the $i$-th entry of $\alpha$ (resp. $\beta$).
We write $\cala = \{ \bfa^{1} ,\ldots , \bfa^{r} \}$ and
assume that $\cala$ is linearly independent.
If $I$ and $J$ are homogeneous ideals of $K[X]$ and $K[Y]$, then the quotient rings $R_1=K[X]/I$ and $R_2=K[Y]/J$ are also multigraded by $\cala$.
Consider the polynomial ring
$$
K[Z] = K[z^{i}_{jk} ~|~ i \in [r] , j \in [\alpha_i] , k \in [\beta_i]]
$$
and consider the ring homomorphism
$$
\phi_{I,J} : K[Z] \rightarrow R_1 \otimes_{K} R_2, \qquad z^{i}_{jk} \mapsto x^{i}_{j} \otimes y^{i}_{k}.
$$
The kernel of $\phi_{I,J}$ is called the {\it toric fiber product} of $I$ and $J$.
It is denoted as $I \times_{\cala} J = \ker (\phi_{I,J})$.
The following result is in \cite[Theorem 12 and Corollary 14]{Sull}.

\begin{proposition}
\proplab{fiber}
Suppose that the set $\cala$ of degree vectors is linearly independent.
Let $\BFF{1}$ and $\BFF{2}$ be homogeneous generating sets for $I$ and $J$, respectively.
Then
$$
\bfn={\rm Lift} (\BFF{1}) \cup
{\rm Lift} (\BFF{2}) \cup
{\rm Quad}_{\cala}
$$
is a homogeneous generating set for $I \times_{\cala} J$.
Moreover, if $\BFF{1}$ and $\BFF{2}$ are Gr\"obner bases of $I$ and $J$, then there exists a monomial order such that $\bfn$ is a Gr\"obner basis for $I \times_{\cala} J$.
The sets ${\rm Lift}(\BFF{1})$, ${\rm Lift}(\BFF{2})$, and ${\rm Quad}_{\cala}$ are defined in \cite{Sull}.
\end{proposition}

On the other hand, if $I$ and $J$ are toric ideals, then $I \times_{\cala} J$ is also a toric ideal.
If $K[S]$ and $K[T]$ are polynomial rings, and
\begin{eqnarray*}
\phi : K[X] & \rightarrow & K[S], \quad x^{i}_{j} \mapsto f^{i}_{j}(S), \\
\psi : K[Y] & \rightarrow & K[T], \quad y^{i}_{k} \mapsto g^{i}_{k}(T)
\end{eqnarray*}
are ring homomorphisms, then we can form the toric fiber product homomorphism
$$
\phi \times_{\cala} \psi : K[Z] \rightarrow K[S,T], \quad z^{i}_{jk} \mapsto f^{i}_{j} (S) g^{i}_{k} (T).
$$
If $I=\ker (\phi)$ and $J=\ker (\psi)$ and both ideals are homogeneous with respect to the grading by $\cala$, then $
I \times_{\cala} J = \ker (\phi \times_{\cala} \psi)
$ (see \cite{EKS}).

\begin{proof}[Proof of \lemref{main2}]
Let $\BFF{1}$ and $\BFF{2}$ be generating sets (resp. Gr\"obner bases) for $J_{\cald_{M_1}}$ and $J_{\cald_{M_2}}$.
From \lemref{main1}, \lemref{gen}, and \propref{fiber}, $\widetilde{\bfn}$ is a generating set (resp. a Gr\"obner basis) for $J_{\widetilde{\cald}}$.
Now we consider two integer matrices
\begin{eqnarray*}
\widetilde{\cald}^{'} &=& 
\left(
\begin{array}{ccc|c|ccc|ccc|c|ccc}
\bfb_1 & & \bfb_1 && \bfb_{n_1} && \bfb_{n_1}&\bfb_1&&\bfb_1&&\bfb_{\gamma_1}&&\bfb_{\gamma_1} \\
\bfd_1 &\cdots &\bfd_{\gamma_2} & \cdots & \bfd_1& \cdots & \bfd_{\gamma_2}& \bfd_1& \cdots& \bfd_{n_2}& \cdots & \bfd_1 & \cdots & \bfd_{n_2} \\
\bfe_1 & & \bfe_1 && \bfe_1 & & \bfe_1 & \bfc_1 && \bfc_{n_2} && \bfc_{1} & & \bfc_{n_2} 
\end{array}
\right), \\
\cald &=& 
\left(
\begin{array}{ccc|c|ccc|ccc|c|ccc}
\bfb_1 & & \bfb_1 && \bfb_{n_1} && \bfb_{n_1}&\bfb_1&&\bfb_1&&\bfb_{\gamma_1}&&\bfb_{\gamma_1} \\
\bfd_1 &\cdots &\bfd_{\gamma_2} & \cdots & \bfd_1 & \cdots & \bfd_{\gamma_2}& \bfd_{\gamma_2 +1}& \cdots& \bfd_{n_2}& \cdots & \bfd_{\gamma_2 +1} & \cdots & \bfd_{n_2} 
\end{array}
\right),
\end{eqnarray*}
where
$$
\bfc_{k} =
\begin{cases}
\bfe_{2} & \text{{\rm if}}~k \in [\gamma_2], \\
\bfe_{1} & \text{{\rm otherwise}}.
\end{cases}
$$
Then $J_{\widetilde{\cald}^{'}} = J_{\widetilde{\cald}}$ because $\widetilde{\cald}^{'}$ can be obtained by an elementary row operation on $\widetilde{\cald}$.
Let $\delta =(0,\ldots,0,-1)\in \zz^{d_1 +d_2 + 2}$.
Since the usual inner product $\delta \cdot (\bfb_{j},\bfd_{k},\bfc_{k})$ equals
$$
\begin{cases}
-1 & \text{{\rm if~}} k \in [\gamma_2], \\
0 & \text{{\rm otherwise,}}
\end{cases}
$$
and $\delta \cdot (\bfb_{j},\bfd_{k},\bfe_{1}) =0$,
it follows that a subring $K[\widehat{Z}]/J_{\cald}$ of $K[Z]/J_{\widetilde{\cald}^{'}}$ is a combinatorial pure subring of $K[Z]/J_{\widetilde{\cald}^{'}}$ (see \cite{Ohsugi}).
Thus $
J_{\cald} = J_{\widetilde{\cald}^{'}} \cap
K[\widehat{Z}].
$
In particular, $N$ is a generating set (resp. a Gr\"obner basis) for $J_{\cald}$.
Furthermore, by elementary row operations on $\cald$, we can obtain the integer matrix arising from $S(M_1, M_2)$ with respect to the basepoint $c$.
Therefore $N$ is a generating set (resp. a Gr\"obner basis) for $J_{S(M_1,M_2)}$.
\end{proof}

\begin{theorem}
Classes $\mqg$ and $\mq$ are closed under series and parallel connections and 2-sums.
\end{theorem}

\begin{proof}
Let $M_1$ and $M_2$ be matroids with $E_1 \cap E_2 = \{ c \}$.
Let $S(M_1 , M_2)$ (resp. $P(M_1 , M_2)$) denote a series (resp. parallel) connection of $M_1$ and $M_2$ with respect to the basepoint $c$.

In the case of series and parallel connections, if $c$ is a loop or a coloop of $M_1$ and the theorem holds for $M_1$ and $M_2$,  then it also holds for $S(M_1,M_2)$ and $P(M_1,M_2)$.
Suppose that neither $M_1$ nor $M_2$ has $c$ as a loop or a coloop.
Then by \lemref{main1} and \lemref{main2}, $\mqg$ and $\mq$ are closed under series connections.
Also, $\mqg$ and $\mq$ are closed under parallel connections from \propref{Blu}, and $P(M_1 , M_2) = [S(M^{\ast}_{1} , M^{\ast}_{2})]^{\ast}$ for any matroids $M_1$ and $M_2$ \cite[Proposition 7.1.14]{Oxley}.

In the case of the 2-sum, since $M_1 \oplus_{2} M_2 = S(M_1 ,M_2) / c$, $\mqg$ and $\mq$ are closed under 2-sums.
\end{proof}

Using the above results, we have
\begin{theorem}
\thmlab{app2}
Let $M$ be a matroid.
If $M$ has no minor isomorphic to any of $M(K_4)$, $\calw^{3}$, $P_6$, or $Q_6$, then the toric ideal $J_{M}$ has a quadratic Gr\"obner basis.
\end{theorem}

\thmref{app2} immediately holds from the following result:

\begin{theorem}
[{\rm \cite[Corollary 3.1]{Chaourar}}]
A matroid $M$ is a minor of direct sums and 2-sums of uniform matroids if and only if $M$ has no minor isomorphic to any of $M(K_4)$, $\calw^{3}$, $P_6$, or $Q_6$.
\end{theorem}

Let $M$ be a matroid on $E$, and let
$$
{\rm rk}:2^{E} \rightarrow \zz_{\ge 0} \quad X \mapsto |B_X|,
$$
where $B_X$ is a basis for $M \setminus (E-X)$.
A function ${\rm rk}$ is said to be the {\it rank function} of $M$.
Let $\lambda_M (X) = {\rm rk} (X) + {\rm rk} (E-X) - {\rm rk}(M)$ for $X \subset E$.
We call $\lambda_M$ the {\it connectivity function} of $M$.
For $X \subset E$, if $\lambda_M (X) < k$, where $k$ is a positive integer, then both $X$ and $(X,E-X)$ are called $k$-{\it separating}.
A $k$-separating pair $(X,E-X)$ for which $\min \{ |X|, |E-X| \} \ge k$ is called a $k$-{\it separation} of $M$ with {\it sides} $X$ and $E-X$.
For all $n \ge 2$, we say that $M$ is $n$-{\it connected} if, for any $k < n$, it has no $k$-separation.

Any matroid which is not 3-connected can be constructed from  3-connected proper minors of itself by a sequence of the operations of direct sums and 2-sums. 
Therefore, in order to prove \conjref{Whiteweak}, it is enough to prove the following conjecture:

\begin{conjecture}
{\rm For any $3$-connected matroid $M$,}
\begin{itemize}
\item[$(1)$] {\rm $J_M$ is generated by quadratic binomial;}
\item[$(2)$] {\rm $J_M$ has a quadratic Gr\"obner basis.}
\end{itemize}
\end{conjecture}

\section*{Acknowledgement}

The author would like to thank Hidefumi Ohsugi for useful comments and suggestions.

\end{document}